\documentclass{au}
\usepackage{amssymb,latexsym}
\usepackage{amsmath}
\usepackage[dvipdfm]{graphicx}
 \usepackage{color}
\usepackage{enumerate}
\newenvironment{enumeratei}{\begin{enumerate}[\upshape (i)]}{\end{enumerate}}
\theoremstyle{plain}
 \newtheorem{theorem}{Theorem}
 
 \newtheorem{proposition}[theorem]{Proposition}
 \newtheorem{corollary}[theorem]{Corollary}
\theoremstyle{definition}
 
 \newtheorem{remark}[theorem]{Remark}
\newcommand \psetn {\textup{Pow}(\set{1,\dots,n})}
\newcommand \alg [1] {\mathfrak{#1}}
\newcommand \tbf [1] {\textbf{#1}}     
\newcommand \width [1] {\textup{width}\,#1}
\newcommand \ejlat [1] { L_{\textup{EJ}}(#1 )}
\newcommand \czlat [1] { L_{\textup{C}}(#1 )}
\renewcommand \emptyset{\varnothing}
\newcommand \upstar [1] {#1^{\ast}}
\newcommand \print [1] {\textup{PrInt}(#1)}
\newcommand \id {\textup{id}}
\newcommand \vpi {\vec \pi}
\newcommand \vsigma {\vec \sigma}
\newcommand \tuple [1] {\langle #1 \rangle}
\newcommand \kset {\set{1,\ldots,k}}

\newcommand \Jir [1] {\textup{Ji}\,#1} 
\newcommand \Mir [1] {\textup{Mi}\,#1}
\newcommand \length [1] {\textup{length}\,#1}
\newcommand \set[1] {\{#1\}}
\newcommand \bigset[1] {\bigl\{#1\bigr\}}
\renewcommand\phi{\varphi}

\newcommand\nothing [1] {}

\begin{document}
\title[Join-distributive lattices by permutations]
{Notes on the description of join-distributive lattices by   permutations}
\author[K.\ Adaricheva]{Kira Adaricheva}
\email{adariche@yu.edu}
\address{Department of Mathematical Sciences, Yeshiva University, New York,  245 Lexington ave., New York, NY 10016,  USA}

\author[G.\ Cz\'edli]{G\'abor Cz\'edli}
\email{czedli@math.u-szeged.hu}
\urladdr{http://www.math.u-szeged.hu/~czedli/}
\address{University of Szeged, Bolyai Institute. 
Szeged, Aradi v\'ertan\'uk tere 1, HUNGARY 6720}

\thanks{This research was supported by the NFSR of Hungary (OTKA), grant numbers  K77432 and
K83219, and by  T\'AMOP-4.2.1/B-09/1/KONV-2010-0005}


\nothing{
05E99 Algebraic combinatorics (1991-now) None of the above, but in this section 
;
52C99  Discrete Geometry  (1991-now) None of the above, but in this section} 


\begin{abstract} Let $L$ be a join-distributive lattice with length $n$ and $\width{(\Jir L)}\leq k$. 
There are two ways to describe $L$ 
by $k-1$ permutations acting on an $n$-element set: a combinatorial way given by P.\,H.~Edelman and  R.\,E.~Jamison  in 1985 and a recent lattice theoretical way of the second author. We  prove that these two approaches are equivalent. Also, we characterize  join-distributive lattices by trajectories.
\end{abstract}

\maketitle

\subsection*{Introduction}
For $x\neq 1$ in a finite lattice $L$, let $\upstar x$ denote the join of upper covers of $x$. A finite lattice $L$ is \emph{join-distributive} if 
the interval $[x,\upstar x]$ is distributive for all $x\in L\setminus\set 1$. For other definitions,  see  K.~Adaricheva~\cite{adaricheva}, 
 K.~Adaricheva, V.A.~Gorbunov and V.I.~Tumanov~\cite{r:adarichevaetal}, and 
N.~Caspard and B.~Monjardet~\cite{caspardmonjardet}, see  G.~Cz\'edli~\cite[Proposition 2.1 and Remark 2.2]{czgcoord} for a recent survey, and see \eqref{Dlwltcsd} before the proof of Corollary~\ref{trajcorol} later for a particularly  useful variant.
The study of (the duals of) join-distributive lattices  goes back to R.\,P.~Dilworth~\cite{r:dilworth40}, 1940. There were  a lot of discoveries and rediscoveries of these lattices and equivalent combinatorial structures; see   \cite{r:adarichevaetal}, \cite{czgcoord}, B.~ Monjardet~\cite{monjardet}, and M.~Stern~\cite{stern} for surveys. 
Note that join-distributivity implies semimodularity; the origin of this result is the combination of M. Ward~\cite{ward} (see also R.\,P.~Dilworth~\cite[page 771]{r:dilworth40}, where \cite{ward} is cited) and S.\,P.~Avann~\cite{avann} (see also 
P.\,H.~Edelman~\cite[Theorem 1.1(E,H)]{edelmanproc}, when \cite{avann} is recalled).

The \emph{join-width} of $L$, denoted by $\width{(\Jir L)}$, is the largest $k$ such that there is a $k$-element antichain of join-irreducible elements of $L$. As usual, $S_n$ stands for the set of permutations acting on the set $\set{1,\ldots,n}$. 
There are two known ways to describe a join-distributive lattice with join-width $k$ and length $n$ by $k-1$ permutations; our goal is to enlighten their connection. This connection exemplifies that Lattice Theory can be applied in Combinatorics and vice versa. 
 We also give a new characterization of join-distributive lattices.

\begin{figure}
\includegraphics[scale=1.0]{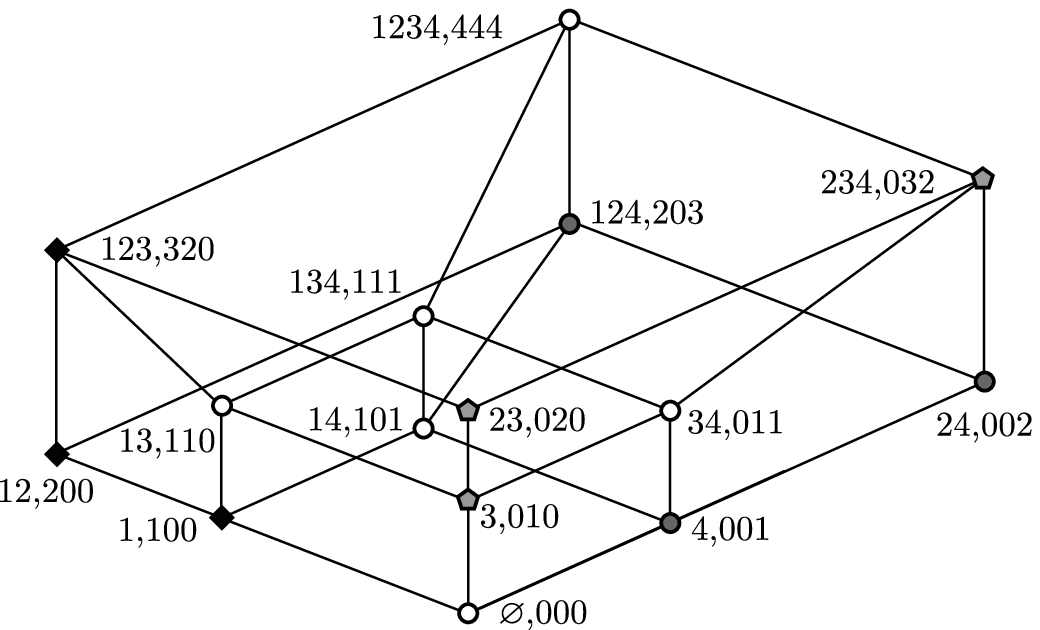}
\caption{An example of $\ejlat{\vsigma}$ and $\czlat\vpi$} 
\label{figone}
\end{figure}

%

\subsection*{Two constructions} For $n\in\mathbb N=\set{1,2,\dots}$ and $k\in\set{2,3,\dots}$, let $\vsigma=\tuple{\sigma_2,\dots,\sigma_k}\in S_n^{k-1}$. For convenience, $\sigma_1\in S_n$ will denote the identity permutation. 
In the powerset join-semilattice 
$\tuple{P(\set{1,\ldots,n});\cup}$, consider the subsemilattice $\ejlat{\vsigma}$ generated by
\begin{equation}\label{gensetEJ}
\bigset{\set{\sigma_i(1),\dots,\sigma_i(j)}: i\in\set{1,\ldots,k}, \,\, j\in\set{0,\dots,n} }\text.
\end{equation}
Since it contains $\emptyset$, $\ejlat{\vsigma}$ is a lattice, the \emph{Edelman-Jamison lattice} determined by $\vsigma$. Its definition above is a straightforward translation from the combinatorial language of P.\,H.~Edelman and R.\,E.~Jamison~\cite[Theorem 5.2]{edeljam} to  Lattice Theory.   Actually, the original version in \cite{edeljam} describes a decomposition of convex geometries.

To present an example, let 
$\sigma_2=\begin{pmatrix}1&2&3&4\cr 3&2&4&1 \end{pmatrix}$ and $\sigma_3=\begin{pmatrix}1&2&3&4\cr 4&2&1&3 \end{pmatrix}$. Then $\vsigma=\tuple{\sigma_2,\sigma_3}\in S_4^2$, and $\ejlat\vsigma$ is depicted in Figure~\ref{figone}. 
In the label of an element in $\ejlat\vsigma$, 
only the part before the comma is relevant;  to save space, subsets are denoted by listing their elements without commas. For example, $134,\!111$ in the figure stands for the subset $\set{1,3,4}$ of $\set{1,2,3,4}$. The chain defined in \eqref{gensetEJ}, apart from its top $\set{1,2,3,4}$ and bottom $\emptyset$,  corresponds to the black-filled small squares for $i=1$, the light grey-filled pentagons for $i=2$, and the dark grey-filled circles for $i=3$. Note that $\ejlat\vsigma$ consists of all subsets of $\set{1,2,3,4}$ but $\set 2$.

Next, we recall a related construction from G.~Cz\'edli~\cite{czgcoord}.  Given $\vpi=\tuple{\pi_{12},\dots, \pi_{1k}}\in S_n^k$, we let $\pi_{ij}=\pi_{1j}\circ\pi_{1i}^{-1}$ for $i,j\in\set{1,\ldots,k}$.
Here we compose permutations from right to left, that is, $(\pi_{1j}\circ\pi_{1i}^{-1})(x) =  \pi_{1j}(\pi_{1i}^{-1}(x))$.
Note that  $\pi_{ii}=\id$, $\pi_{ij}=\pi_{ji}^{-1}$, and $\pi_{jt}\circ\pi_{ij}=\pi_{it}$ hold for all $i,j,t\in\set{1,\dots,k}$.
By an \emph{eligible $\vpi$-tuple} we mean a $k$-tuple $\vec x=\langle x_1,\ldots, x_k\rangle \in\set{0,1,\ldots,n}^k$ such that $\pi_{ij}(x_i+1 )\geq x_j+1$ holds for all $i,j\in \kset $ such that $x_i<n$. Note that an eligible $\vpi$-tuple belongs to 
$\set{0,1,\dots, n-1}^k\cup\set{\tuple{n,\dots,n}}$ since $x_j=n$ implies $x_i=n$.
The set of eligible $\vpi$-tuples is denoted by $\czlat{\vpi}$. 
It is a poset with respect to the componentwise order: $\vec x\leq \vec y$ means that $x_i\leq y_i$ for all $i\in \kset $. It is trivial to check that $\tuple{n,\dots,n}\in\czlat{\vpi}$ and that 
$\czlat{\vpi}$ is a meet-subsemilattice of the $k$-th direct power of the chain $\set{0\prec 1\prec\dots\prec n}$. Therefore, $\czlat{\vpi}$ is a lattice, the \emph{$\vpi$-coordinatized lattice}.
Its construction is motivated by 
G.~Cz\'edli and E.\,T.~Schmidt~\cite[Theorem 1]{czgschthowto}, see also M.~Stern~\cite{stern}, which asserts that there is a surjective cover-preserving join-homomorphism $\phi\colon\set{0\prec\cdots \prec n}^k\to L$, provided $L$ is semimodular. Then, as it is easy to verify, $u\mapsto \bigvee\set{x: \phi(x)=u}$ is a meet-embedding of $L$ into $\set{0\prec\cdots \prec n}^k$.

To give an example, let 
$\pi_{12}=\begin{pmatrix}1&2&3&4\cr 4&2&1&3 \end{pmatrix}$,  $\pi_{13}=\begin{pmatrix}1&2&3&4\cr 3&2&4&1 \end{pmatrix}$, and let  $\vpi=\tuple{\pi_{12},\pi_{13}}\in S_4^2$. 
Then  Figure~\ref{figone} also gives $\czlat\vpi$; the eligible $\vpi$-tuples are given after the commas in the labels. For example, $23,\!020$ in the figure corresponds to $\tuple{ 0,2,0}$.  Note that if $\mu_{12}=\begin{pmatrix}1&2&3&4\cr 3&2&1&4 \end{pmatrix}$ and  $\mu_{13}=\pi_{13}$, then 
$\czlat\vpi\cong \czlat{\vec \mu}$. 
Furthermore, the  problem of characterizing those pairs of members of $S_n^k$ that determine the same lattice is not solved yet if $k\geq 3$. For $k=2$ the solution is given in G.~Cz\'edli and E.\,T.~Schmidt~\cite{czgschperm}; besides $\czlat\vpi\cong \czlat{\vec \mu}$ above, see  also G.~Cz\'edli~\cite[Example 5.3]{czgcoord} to see the difficulty.

The connection of join-distributivity to $\ejlat{\vpi}$  and $\czlat{\vpi}$ will be given soon.

\subsection*{The two constructions are equivalent} For $\tuple{\gamma_2,\dots,\gamma_k}\in S_n^{k-1}$, we let   $\tuple{\gamma_2,\dots,\gamma_k}^{-1}=\tuple{\gamma_2^{-1},\dots,\gamma_k^{-1}}$.

\begin{proposition}\label{ourprop} For every $\vsigma\in S_n^{k-1}$, $\ejlat{\vsigma}$ is isomorphic to $\czlat{\vsigma^{-1}}$.
\end{proposition}

In some vague sense, Figure~\ref{figone} reveals  why $\ejlat{\vsigma}$ could be of the form $\czlat{\vpi}$ for some $\vpi$. Namely,  for $x\in \ejlat\vsigma$ and $i\in \set{1,\dots, k}$, we can define the $i$-th coordinate of $x$ as the length of the intersection of the ideal $\set{y\in \ejlat\vsigma: y\leq x  }$ and the chain given in \eqref{gensetEJ}. However, the proof is more complex than this initial idea.

\begin{proof} Denote $\vsigma^{-1}$ by 
$\vpi=\tuple{\pi_{12},\dots,\pi_{1k}}$. 
Note that $\pi_{11}=\sigma_1^{-1}=\id\in S_n$. 
For $U\in \ejlat{\vsigma}$ and $i\in\set{1,\dots,k}$, let $U(i)=\max\bigset{j: \set{\sigma_i(1),\dots,\sigma_i(j)}\subseteq U}$, where $\max\emptyset$ is defined to be $0$.
We assert that the map
\[
\text{$\phi\colon \ejlat{\vsigma}\to   \czlat{\vpi}$, defined by  $U\mapsto\tuple{U(1),\dots,U(k)}$},
\]
is a lattice isomorphism. To prove that $\phi(U)$ is an eligible $\vpi$-tuple, assume that $i,j\in\set{1,\ldots, k}$ such that $U(i)<n$. Then $\sigma_i(U(i)+1)\notin U$ yields $\sigma_i(U(i)+1)\notin\set{\sigma_j(1),\dots,\sigma_j(U(j))}$. However, $\sigma_i(U(i)+1)\in\set{1,\dots,n}=\set{\sigma_j(1),\dots,\sigma_j(n)}$, and we conclude that $\sigma_i(U(i)+1)=\sigma_j(t)$ holds for some   $t\in\set{U(j)+1,\dots, n}$.
Hence 
\begin{align*}
\pi_{ij}(U(i)+1)&= (\pi_{1j}\circ\pi_{i1})(U(i)+1) =  \pi_{1j}\bigl(\pi_{i1}(U(i)+1)\bigr)\cr
&= \pi_{1j}\bigl(\pi_{1i}^{-1}(U(i)+1)\bigr) = \sigma_{j}^{-1}\bigl(\sigma_i(U(i)+1)\bigr) \cr
& = \sigma_{j}^{-1}\bigl(\sigma_j(t)\bigr) = t \geq U(j)+1\text. 
\end{align*}
This proves that $\phi(U)$ is an eligible $\vpi$-tuple, and $\phi$ is a map from $\ejlat{\vsigma}$ to $\czlat{\vpi}$.
Since $\ejlat{\vsigma}$ is generated by the set given in \eqref{gensetEJ}, we conclude 
\[
U=\bigcup_{i=1}^k \set{\sigma_i(1),\dots,\sigma_i(U(i))}\text{.}
\]
This implies that $U$ is determined by $\tuple{U(1),\dots, U(k)}=\phi(U)$, that is, $\phi$ is injective. To prove that $\phi$ is surjective, let $\vec x=\tuple{x_1,\dots,x_k}$ be a $\vpi$-eligible tuple, that is, $\vec x\in \czlat{\vpi}$.
Define
\begin{equation}\label{invfi} V=\bigcup_{i=1}^k\set{\sigma_i(1),\dots,\sigma_i(x_i)}\text.
\end{equation}
(Note that if $x_i=0$, then $\set{\sigma_i(1),\dots,\sigma_i(x_i)}$ denotes the empty set.) For the sake of contradiction, suppose $\phi(V)\neq  \vec x$.  Then, by the definition of $\phi$, there exists an $i\in\set{1,\dots,k}$ such that $\sigma_i(x_i+1)\in V$. Hence,  there is a $j\in\set{1,\dots,k}$ such that 
$\sigma_i(x_i+1)\in \set{\sigma_j(1),\dots,\sigma_j(x_j)}$. That is, $\sigma_i(x_i+1)=\sigma_j(t)$ for some $t\in\set{1,\dots,x_j}$. Therefore,
\begin{align*}\pi_{ij}(x_i+1)&=\pi_{1j}(\pi_{i1}(x_i+1))=\pi_{1j}(\pi_{1i}^{-1}(x_i+1)
=\sigma_{j}^{-1}(\sigma_i (x_i+1)) \cr
&=\sigma_{j}^{-1}(\sigma_j(t))=t\leq x_j,
\end{align*}
which contradicts the $\vpi$-eligibility of $\vec x$. Thus $\phi(V)=\vec x$ and $\phi$ is surjective. 

We have shown that $\phi$ is bijective. For $\vec x\in \czlat\vpi$, $\phi^{-1}(\vec x)$ is the set $V$ given in \eqref{invfi}. Thus $\phi$ and $\phi^{-1}$ are monotone, and $\phi$ is a lattice isomorphism.
\end{proof}

\subsection*{Two descriptions}
The following theorem is a straightforward consequence of Theorems 5.1 and 5.2 in P.\,H.~Edelman and R.\,E.~Jamison~\cite{edeljam}, which were formulated and  proved within Combinatorics.

\begin{theorem}\label{thmEJ} Up to isomorphism, join-distributive lattices of length $n$ and join-width at most $k$ are characterized as lattices $\ejlat{\vsigma}$ with $\vsigma\in S_n^{k-1}$.
\end{theorem}

The next theorem was motivated and proved by the second author~\cite{czgcoord} in a purely lattice theoretical way. 

\begin{theorem}\label{thmCz} Up to isomorphism, join-distributive lattices of length $n$ and join-width at most $k$ are characterized as the $\vpi$-coordinatized lattices $\czlat{\vpi}$ with $\vpi\in S_n^{k-1}$.
\end{theorem}
\begin{remark}\label{rfnmany} Since there is no restriction on $(n,k)\in \mathbb N\times\set{2,3,\dots}$ in Theorems~\ref{thmEJ} and \ref{thmCz}, one might have the feeling that, for a given $n$, the join-width of a join-distributive lattice of length $n$ can be arbitrarily large. This is not so since, up to isomorphism, there are only finitely many join-distributive lattices of length $n$.
\end{remark}

The statement of Remark~\ref{rfnmany} follows from the fact that each join-distributive lattice of length $n$ is dually isomorphic to the lattice of closed sets of a convex geometry on the set $\set{1,\dots,n}$, see P.\,H.~Edelman \cite[Theorem 3.3]{edelman} together with the 
 sixteenth line in the proof of Theorem 1.9 in  K.~Adaricheva, V.A.~Gorbunov and V.I.~Tumanov~\cite{r:adarichevaetal}; see also 
\cite[Lemma 7.4]{czgcoord}, where this is surveyed. The statement also follows, in a different way, from \cite[Corollary 4.4]{czgcoord}.

\begin{remark} Obviously, Proposition~\ref{ourprop} and  Theorem~\ref{thmEJ} imply Theorem~\ref{thmCz} and, similarly, Proposition~\ref{ourprop} and  Theorem~\ref{thmCz} imply Theorem~\ref{thmEJ}. Thus we obtain a  new, combinatorial proof of Theorem~\ref{thmCz} and a new, lattice theoretical proof of Theorem~\ref{thmEJ}.
\end{remark}

\subsection*{Comparison} We can compare Theorems~\ref{thmEJ} and \ref{thmCz}, and the corresponding original approaches, as follows.

In case of Theorem~\ref{thmEJ}, the construction of the lattice $\ejlat\vsigma$ is very simple, and a join-generating subset is also given. 

In case of Theorem~\ref{thmCz}, the elements of the lattice $\czlat\vpi$ are exactly given by their coordinates, the eligible $\vpi$-tuples. Moreover, the meet operation is easy, and we have a satisfactory description of the optimal meet-generating subset 
since it was proved in \cite[Lemma 6.5]{czgcoord} that 
\[\Mir{\!(\czlat\vpi)}=\bigset{\tuple{\pi_{11}(i)-1,\dots, \pi_{1k}(i)-1}: i\in\set{1,\dots, n} }\text.  
\]

\subsection*{Characterization by trajectories}
For a lattice $L$ of finite length, the set 
$ \bigset{[a,b]: a\prec b,\,\, a,b\in L}$ of prime intervals of $L$ will be denoted by $\print L$. For $[a,b],[c,d]\in\print L$, we say that $[a,b]$ and $[c,d]$ are \emph{consecutive} if $\set{a,b,c,d}$ is a covering square, that is, a 4-element cover-preserving boolean sublattice of $L$. The transitive reflexive closure of the consecutiveness relation on $\print L$ is an equivalence, and the blocks of this equivalence relation are called the \emph{trajectories} of $L$; this concept was introduced for some particular semimodular lattices in G.~Cz\'edli and E.\,T.~Schmidt~\cite{czgschtJH}.  For distinct $[a,b],[c,d]\in\print L$, these two prime intervals are \emph{comparable} if either $b\leq c$, or $d\leq a$.
Before formulating the last statement of the paper, it is reasonable to mention that, for any finite lattice $L$, 
\begin{equation}\label{Dlwltcsd}
\text{$L$ is join-distributive if{f} it is semimodular and meet-semidistributive.}
\end{equation}
This follows from K.~Adaricheva, V.A.~Gorbunov and V.I.\ Tumanov~\cite[Theorems 1.7 and 1.9]{r:adarichevaetal};
see also  D.~ Armstrong~\cite[Theorem 2.7]{armstrong} for the present formulation.

\begin{corollary}\label{trajcorol}
For a semimodular lattice $L$,  the following three conditions are equivalent.
\begin{enumeratei}
\item\label{Dlwltcsa} $L$ is join-distributive.
\item\label{Dlwltcsb} $L$ is of finite length, and for every  trajectory $T$ of $L$ and every maximal chain $C$ of $L$, $|\print C\cap T|=1$. 
\item\label{Dlwltcsc} $L$ is of finite length, and no two distinct comparable prime intervals of $L$ belong to the same trajectory.
\end{enumeratei}
\end{corollary}

As an interesting consequence, note that each of \eqref{Dlwltcsb} and  
\eqref{Dlwltcsc} above, together with semimodularity, implies that $L$ is finite.

\begin{proof}[Proof of Corollary~\ref{trajcorol}] 
Since any two comparable prime intervals belong to the set of prime intervals of   an appropriate maximal chain $C$,  \eqref{Dlwltcsb} implies \eqref{Dlwltcsc}. So we have to prove that
 \eqref{Dlwltcsa} $\Rightarrow$ \eqref{Dlwltcsc} and that \eqref{Dlwltcsc} $\Rightarrow$ \eqref{Dlwltcsa}; we give two alternative arguments for each of these two implications. Let $n=\length L$.

Assume \eqref{Dlwltcsa}. Then $L$ is semimodular by \eqref{Dlwltcsd}, and  it contains no cover-preserving diamond by the definition of join-distributivity. Thus G.~Cz\'edli~\cite[Lemma 3.3]{czgcoord} implies \eqref{Dlwltcsb}. 

For a second argument, assume \eqref{Dlwltcsa} again. Let $\psetn$ denote the set of all subsets of $\set{1,\dots,n}$. 
It is known that $L$ is isomorphic to an appropriate join-subsemilattice $\alg F$ of the powerset $\bigl(\psetn;\cup\bigr)$ such that  $\emptyset\in\alg F$ and  each $X\in\alg F\setminus\set\emptyset$ contains an element $a$ with the property $X\setminus\set a\in \alg F$. The structure $\tuple{\set{1,\dots,n}; \alg F}$ is an \emph{antimatroid} on the base set $\set{1,\dots,n}$ (this concept is due to 
R.\,E.~Jamison-Waldner~\cite{jamisonwaldner}), and the existence of an appropriate $\alg F$
follows from P.\,H.~Edelman~\cite[Theorem 3.3]{edelman} and D.~Armstrong~\cite [Lemma 2.5]{armstrong}; see also 
 K.~Adaricheva, V.A.~Gorbunov and V.I.~Tumanov~\cite[Subsection 3.1]{r:adarichevaetal} and G.~Cz\'edli~\cite[Section 7]{czgcoord}. Now, we can assume that $L=\alg F$. We assert that, for any $X,Y\in \alg F$,
\begin{equation}\label{pnthmk}
X\prec Y \quad\text{if{f}}\quad X\subset Y\text{ and }|Y\setminus X|=1\text.
\end{equation}
The ``if'' part is obvious. For the sake of contradiction, suppose $X\prec Y$ and $x$ and $y$ are distinct elements in $Y\setminus X$. Pick a sequence $Y=Y_0\supset Y_1\supset\dots\supset Y_t=\emptyset$ in $\alg F$ such that $|Y_{i-1}\setminus Y_i|=1$ for $i\in\set{1,\dots,t}$. Then there is a $j$ such that $|Y_j\cap\set{x,y}|=1$. This gives the desired contradiction since $X\cup Y_j\in\alg F$ but $X\subset X\cup Y_j\subset Y$.

Armed with \eqref{pnthmk}, assume that $\set{A=B\wedge C, B,C,D=B\cup C}$ is a covering square in $\alg F$. Note that $A$ and $B\cap C$ can be different; however, $A\subseteq B\cap C$. By \eqref{pnthmk}, there exist $u,x\in D$ such that 
$B=D\setminus \set{u}$ and $C=D\setminus \set{x}$. These elements are distinct since $B\neq C$. Hence $x\in B$ and, by $A\subseteq   C$, $x\notin A$. Using \eqref{pnthmk} again, we obtain $A=B\setminus\set x$. We have seen that whenever $[A,B]$ and $[C,D]$ are consecutive prime intervals, then there is a common $x$ such that $A=B\setminus\set x$ and $C=D\setminus\set x$. This implies that for each trajectory $T$ of $\alg F$, there exists an $x_T\in \set{1,\dots,n}$ such that $X=Y\setminus\set {x_T}$ holds for all $[X,Y]\in T$. Clearly, this implies that \eqref{Dlwltcsc} holds for $\alg F$, and also for $L$.

Next, assume \eqref{Dlwltcsc}.  Since any two prime intervals of a cover-preserving diamond would belong to the same trajectory, $L$ contains no such diamond.  Again, there are two ways to conclude \eqref{Dlwltcsa}.   

First, by \cite[Proposition 6.1]{czgcoord}, 
$L$ is isomorphic to $\czlat\vpi$ for some $k$ and $\vpi\in S_n^{k-1}$, and we obtain from Theorem~\ref{thmCz} that \eqref{Dlwltcsa} holds. 

Second, H.~Abels~\cite[Theorem 3.9(a$\Rightarrow$b)]{abels} implies that $L$ is a cover-preserving join-subsemilattice of a finite distributive lattice $D$. Thus if $x\in L\setminus\set 1$, then the interval $[x,\upstar x]_L$ of $L$ is a cover-preserving join-subsemilattice of $D$. Let $a_1,\dots,a_t$ be the covers of $x$ in $L$, that is, the atoms of $[x,\upstar x]_L$. If we had, say, $a_1\leq a_2\vee \dots\vee a_t$, then we would get a contradiction  in $D$ as follows: $a_1=a_1\wedge (a_2\vee\dots \vee a_t)=(a_1\wedge a_2)\vee\dots\vee (a_1\wedge a_t)=x\wedge\dots\wedge x=x$. Thus  $a_1,\dots,a_t$ are independent atoms in $[x,\upstar x]_L$. Therefore, it follows from G.~Gr\"atzer~\cite[Theorem 380]{GGLT} and the semimodularity of $[x,\upstar x]_L$ that the sublattice $S$ generated by $\set{a_1,\dots, a_t}$ in $L$ is  the $2^t$-element boolean lattice. In particular, $\length{S}=t =\length{\bigl([x,\upstar x]_L\bigr)}$ since  $\set{x,\upstar x}\subseteq S\subseteq  [x,\upstar x]_L$. Since the embedding is cover-preserving, the length of the interval $[x,\upstar x]_D$ in $D$ is also $t$. Hence $|\Jir{([x,\upstar x]_D)}|=t$ by   \cite[Corollary 112]{GGLT}, which clearly implies  $|[x,\upstar x]_D|\leq 2^t$. Now from $ [x,\upstar x]_L\subseteq [x,\upstar x]_D$ and $ 2^t=|S|\leq |[x,\upstar x]_L|\leq 
|[x,\upstar x]_D|\leq 2^t$ we conclude $[x,\upstar x]_L=[x,\upstar x]_D$. This  implies that $[x,\upstar x]_L$ is distributive. Thus \eqref{Dlwltcsa} holds.
\end{proof}

\end{document}